\documentclass[journal,twoside,web]{ieeecolor}

\usepackage{generic}
\usepackage{cite}
\usepackage{amsmath,amssymb,amsfonts}
\usepackage{algorithmic}
\usepackage{graphicx}
\usepackage{algorithm,algorithmic}
\usepackage{hyperref}
\hypersetup{hidelinks=true}
\usepackage{textcomp}
\usepackage{enumerate}
\newtheorem{lemma}{Lemma}
\newtheorem{remark}{Remark}
\newtheorem{corollary}{Corollary}

\usepackage{listings}
\usepackage{booktabs}
\newcommand{\vct}[1]{\mathbf{#1}}
\newcommand{\norm}[1]{\left\lVert#1\right\rVert}

\usepackage{orcidlink}

\def\BibTeX{{\rm B\kern-.05em{\sc i\kern-.025em b}\kern-.08em
    T\kern-.1667em\lower.7ex\hbox{E}\kern-.125emX}}
\begin{document}
\title{Static Output Feedback Control for Multi-Agent Systems Using Nash Equilibrium}
\author{Paul Serna-Torre, \IEEEmembership{Member, IEEE} \orcidlink{0000-0003-0578-7933}, and Patricia Hidalgo-Gonzalez, \IEEEmembership{Member, IEEE} \orcidlink{0000-0003-3635-8285}
\thanks{
Corresponding authors:
Paul Serna-Torre, Patricia Hidalgo-Gonzalez. Emails: 
        {psernatorre@ucsd.edu, phidalgogonzalez@ucsd.edu}}
\thanks{Paul Serna-Torre and Patricia Hidalgo-Gonzalez are with the Department of Mechanical and Aerospace Engineering and with the Center for Energy Research, University of California, San Diego, CA, 92093, USA.}
\thanks{This research is supported by the Center for Energy Research at the University of California San Diego, and the Ministry of Education of Peru.}
\thanks{Paul Serna-Torre is also supported by a scholarship from the Ministry of Education of Peru.}}

\maketitle

\begin{abstract}
This work proposes a semidefinite programming (SDP) model that synthesizes static output control feedback for networked linear time-invariant multi-agent systems (MAS). We theoretically prove that the output feedback control of MAS can be constructed using the Nash equilibrium previously derived for state-feedback control. The Nash equilibrium considers that agents seek to minimize their individual cost functions subject to the system dynamics of the MAS. The resulting control from the proposed SDP model ensures the stability of the closed-loop system and dissipativity when a disturbance is present in the MAS. A set of four case studies is implemented. The first three numerical cases serve to validate and compare cost effectiveness and time-domain performance of the proposed control framework with other state-of-the-art frameworks. The fourth case designs a frequency regulation scheme for a modified IEEE 24-bus power grid.  Using the proposed framework, distributed controllers are synthesized for the inverters which are considered agents in the grid. Each inverter uses partial state information of the neighboring inverters. The resulting control scheme is implemented in a high-fidelity non-linear model of 261 states that describes the electromagnetic transient dynamics of the power grid. Hence, compared to prior studies that usually validate their frameworks in small-scale numerical models, this fourth case demonstrates a virtue of our proposed SDP: its scalability and practical utility in real-world networked MAS. We implement the proposed framework in an open-access toolbox in Python and MATLAB, enabling practitioners and researchers to readily implement it in their applications. 



\end{abstract}

\begin{IEEEkeywords}
Differential game theory, linear systems theory, Nash equilibrium, output feedback, power systems.
\end{IEEEkeywords}

\section{Introduction}
\label{sec_introduction}

A wide range of real-world dynamical systems, such as power grids and robot systems, can be studied as multi-agent systems (MAS) \cite{interconnect_dyn_gianluca}. A MAS is composed by agents that can be connected through couplings or links. The agents are decision makers that perform a global task while seeking to optimize their individual performance metrics. Non-cooperative linear quadratic differential games (NLQGs), since its inception in the 1950s, has been increasingly developed to provide a framework for the control of MAS \cite{review_games_grid, tutorial_games_motion}. 


In NLQGs, each agent pursues to minimize its individual quadratic cost function by choosing its own strategy subject to a linear time-invariant (LTI) system state-space model \cite{Engwerda_LQbook}. Hence, each agent is considered self-interested as it does not cooperate in minimizing neither other agents' cost functions nor a global cost function. Despite this non-cooperative aspect, the strategies of the agents can be designed to ultimately accomplish a global task for the MAS. For instance, the work in \cite{noncooperative_paulpatricia} implements a control scheme using NLQGs to control the dynamics of multiple power inverters in a microgrid. Although the power inverters are considered selfish, the strategies for the inverters are designed to satisfy a regulation service and the microgrid's load. Moreover, in NLQGs, the system state-space representation does not only model the dynamics of the agents, but also it can incorporate the dynamics of the couplings between the agents. Hence, NLQGs can be considered as a multi-objective optimization problem subject to the dynamics of agents and their couplings. Given this complexity in NLQGs, the challenge is how to determine the strategies for the agents. 

Several works have investigated the determination of these strategies using the concept of Nash equilibrium. As the studies in \cite{lineargames_engwerda, thulasi_constructive_nash} show, these strategies can be determined in a static linear state-feedback form by solving coupled algebraic Riccati equations (CAREs). However, as CAREs are bilinear, solving them is not trivial. The works in \cite{capello_nash_approx} and \cite{Engwerda2007} propose relaxed versions of the CAREs and iterative algorithms, respectively, to solve them. Inspired by potential differential games \cite{potgames_fonseca_morales}, the work in \cite{scarpa_pote_games} reformulates the NLQG as a linear quadratic regulator (LQR) problem under certain assumptions, which is  easier to solve. An underlying assumption across the works in \cite{lineargames_engwerda, thulasi_constructive_nash, capello_nash_approx, Engwerda2007, scarpa_pote_games} is that the resulting Nash strategies use complete state information of the agents and couplings. However, this assumption makes the implementation of full state-feedback strategies impractical in several real-world networked systems. Agents can be placed in distant locations, leading to
considerable time delays in the communication flow. Thus, opting for distributed strategies is suggested over full state-feedback approaches \cite{networkedsystems_com_review}. In other cases, the access to the state information of the MAS requires extensive communication efforts. Aside from these communication infrastructure issues, privacy concerns may refrain agents from sharing their complete state information. 

These challenges to implement Nash state-feedback strategies can be exemplified as follows. In \cite{paulpatricia_nash_freq}, 
Nash state-feedback strategies are designed for the regulation of electrical frequencies in a 9-bus power grid. However, its implementation in a real-life large-scale grid can result impractical. This is because measuring frequencies and voltage angles, which both are the states in the grid, would require to install Phasor Measurement Units (PMU) at every substation in the grid. However, given the high cost of PMUs, installing numerous PMUs is financially nonviable. Hence, as Nash state-feedback strategies can be difficult or impractical to implement, it opens the question on how to synthesize strategies in NLQGs that use partial state information.  
In efforts to fill this gap, the work in \cite{lin_pursuitgame, wang_distributed_pursuer} determines Nash strategies for a pursuit-evasion game with limited state observation among agents. In a more general setting, the authors in \cite{thulasi_distributedgames, onlythulasi_game_distributed} propose a communication graph and bilinear matrix equalities to synthesize Nash strategies that use neighboring state information. The studies in \cite{capello_distributednash, capello_distributednash_2} demonstrate that when the communication graph in \cite{thulasi_distributedgames} is arborescence or directed acyclic, the CAREs reduce to a set of standard algebraic Riccati equations (AREs), which synthesizes distributed strategies for the agents. Although \cite{thulasi_distributedgames} and \cite{onlythulasi_game_distributed} develop careful theoretical results, they do not provide a method to explicitly solve its proposed bilinear matrix equations for an arbitrary number of agents. Moreover, the mentioned frameworks in \cite{lin_pursuitgame, wang_distributed_pursuer, onlythulasi_game_distributed, capello_distributednash, capello_distributednash_2} consider that the dynamics of the agents are uncoupled. This implies that the state matrix of the MAS state-space model is a block diagonal matrix that stacks the state matrices of the agents' state-space models. However, this assumption does not hold for several real-world systems, where a physical network inevitably connects agents' states through physical laws. For instance, in a gas network \cite{gas_paper}, the gas flowing between nodes is a function of the nodal pressures, which are states that can be associated to gas producers (agents) at the nodes. 


Furthermore, the article in \cite{engwerda_outputnash} derives sufficient and necessary conditions for the existence of Nash strategies with static output feedback structure in NLQGs. However, fulfilling these conditions can be impossible, as it is demonstrated in a three-agent test system \cite{distributednash_roy}. As an alternative to Nash output strategies, the works in \cite{distributednash_roy, Roy_outputnash} develop the concept of network adapted feedback guaranteed cost equilibrium. Using this concept, the work in \cite{distributednash_roy} proposes a iterative algorithm that incorporates Linear Matrix Inequalities (LMIs) to synthesize the output feedback strategies. The work in \cite{Roy_outputnash} thoroughly extends this framework by incorporating unknown disturbances in the MAS dynamics. Although the sophisticated algorithms in \cite{distributednash_roy, Roy_outputnash} can be implementable through for-loops and conditional statements, a more straightforward approach, such as a single optimization problem, would be more efficient to solve for the synthesis of output feedback control and, to the best of our knowledge, there is no such formulation yet.  Hence, in MAS with dynamically coupled and heterogeneous agents, there is still a gap on how to synthesize control strategies that use partial state information. 

Moreover, a common practice seen across the aforementioned works in \cite{capello_nash_approx, Engwerda2007, distributednash_roy, Roy_outputnash, lin_pursuitgame, thulasi_distributedgames, capello_distributednash, capello_distributednash_2} is the implementation of their proposed frameworks in small numerical examples, e.g., third-order model with five agents, leaving the scalability of their frameworks unexplored. Hence, it is an open question whether their proposed frameworks can be solved or converge for more realistic MAS, which are modeled by high-order dynamical systems with multiple agents. Aside from this scalability question, the prior work and well-established software such as MATLAB fall short of providing a publicly available toolkit for the control synthesis of MAS from a differential game perspective. Built-in functions in Python and MATLAB for traditional control synthesis such as LQR and $H_{\infty}$ are offered; however, packages for control synthesis for MAS using differential games are limited. Filling this gap would substantially facilitate the further investigation and implementation of differential games in practical industry and research settings.   





\section{Contributions}
The major contributions of this work are the following:
    \begin{enumerate}
        \item We propose a novel SDP model that synthesizes a set of linear feedback control strategies using a Nash equilibrium for a LTI networked MAS with restrictive communication. As the proposed SDP uses a Nash equilibrium solution, the resulting strategies aim to capture the self-interested nature of the agents in NLQGs. Moreover, the resulting strategies ensure the stability of the closed-loop system and dissipativity when a disturbance is present in the MAS. Compared to \cite{thulasi_distributedgames, onlythulasi_game_distributed, engwerda_outputnash, distributednash_roy, Roy_outputnash} which propose iterative methods or do not explicitly solve the equations for the synthesis of output feedback strategies in NLQGs, we put forth a SDP model which results more straight-forward to solve using off-the-shelf solvers. This proposed SDP can be also applied to a MAS with dynamically coupled agents, a scenario not accommodated in previous works \cite{lin_pursuitgame, wang_distributed_pursuer, onlythulasi_game_distributed, capello_distributednash, capello_distributednash_2}.
        \item We showcase the time-domain performance and cost-effectiveness of the proposed control frameworks  compared with other control frameworks for MAS. 
        Unlike previous works \cite{capello_nash_approx, Engwerda2007, distributednash_roy, Roy_outputnash, lin_pursuitgame, thulasi_distributedgames, capello_distributednash, capello_distributednash_2} that implement their framework in small-scale cases, our work tests the scalability of the proposed framework in a more realistic setting. To this end, we  design
        a distributed control scheme for frequency regulation in a modified IEEE 24-bus power grid. Subsequently, we implement the designed control scheme in a high-fidelity model of the power grid, which models the non-linear electromagnetic transient (EMT) dynamics of the grid's components. Furthermore, we compare the controller with classical decentralized frequency regulation strategies. Through the implementation in a realistic MAS, our work does not remain theoretical in nature, but provide practical contributions.
        \item We provide an open-access toolbox of the proposed control framework. Practitioners and researchers can use this toolkit to readily implement and benchmark the proposed control framework. 
    \end{enumerate}

\section{Nomenclature}
Let $\mathbb{R}$ and $\mathbb{C}$ denote the set of real and complex numbers, respectively. The notation $|\mathcal{C}|$ indicates the cardinality of the set $\mathcal{C}$. The matrices $0_{m \times n}$ and $I_{n}$ denote the zero matrix of dimension $m \times n$ and the identity matrix of dimension $n \times n$, respectively. A scalar, a vector, and a matrix are denoted by $x$, $\vct{x}$ and $X$, respectively. The construct 
$B=$ diag($A_1, ..., A_n$) returns the block diagonal matrix $B$ by aligning the  matrices $A_1, ..., A_n$ along the diagonal of $B$. A norm is denoted by $\norm{ \cdot }$. $L_{2(a,b)}$ represents the set of square integrable functions, i.e., $\int_a^b |f(t)|^2 dt < \infty$. The dependence on time is omitted when it can
be inferred from the context, e.g., $\vct{x}(t)$ is
equivalent to $\vct{x}$. 
\section{Problem Statement}
\label{sec_problem_statement}
Let the dynamics of a MAS be described by the following LTI system:
    \begin{align}
        \dot{{\vct{{x}}}} & = A\vct{x} + \sum_{i \in \mathcal{C}}B_i \vct{u}_i, \label{eq_mas_dyn1} \\
        \vct{y}_{i} & = C_{i} \vct{x} + D_{i}\vct{d}_{i}, \label{eq_mas_dyn2}
    \end{align} 
where $\mathcal{C} =\{1, ..., |\mathcal{C}| \}$ represents the set of agents, $\vct{x} \in \mathbb{R}^n$ contains the states of the dynamics of the agents and couplings, $\vct{u} _i \in \mathbb{R}^{m_i}$ is the input of agent $i \in \mathcal{C}$, $\vct{y}_i \in \mathbb{R}^{p_i}$ is the output of agent $i \in \mathcal{C}$, and $\vct{d}_i \in \mathbb{R}^{m_i}$ is a disturbance to the output of agent $i \in \mathcal{C}$. Thus, $A \in \mathbb{R}^{n \times n}$, $B_i \in \mathbb{R}^{n \times m_i}$, $C_i \in \mathbb{R}^{p_i \times n}$, and $D_i \in \mathbb{R}^{p_i \times m_i}$. The LTI system in \eqref{eq_mas_dyn1}-\eqref{eq_mas_dyn2} has the following characteristics: 
\begin{enumerate}[(C1)]
\item the matrix pairs $(A, \begin{bmatrix}
    B_1 & ... & B_{|\mathcal{C}|}
\end{bmatrix})$ and $(A, diag(C_1, ..., C_{|\mathcal{C}|})$ are stabilizable and detectable, respectively,
\item $A$ is not necessarily a block diagonal matrix, thus the dynamics of the agents can be coupled,
\item the entries of $A$ can be different, then the agents and couplings can have heterogeneous dynamics,
\item $C_i \ \forall i \in \mathcal{C}$ is not necessarily equal to $I_n$, thus the agents have partial access to the system state $\vct{x}$, and
\item the disturbance has finite energy, i.e., $\vct{d}_i \in L^{m_i}_{2[0, \infty)} \forall i \in \mathcal{C}$.
\end{enumerate}
Furthermore, each agent $i \in \mathcal{C}$ is endowed with a quadratic cost function:
\begin{equation} \label{eq_functional_player}
   f_i(\vct{x}_0, \vct{u}_1, ..., \vct{u}_{|\mathcal{C}|}) = \int_0^{\infty}\vct{x}^\top Q_i \vct{x} + \vct{u}^\top_i R_i \vct{u}_i dt,
\end{equation}
where $Q_i \succeq 0$, $R_i \succ 0$, and $\vct{x}_0$ is the initial state for \eqref{eq_mas_dyn1}-\eqref{eq_mas_dyn2}. The objective is to determine the set of static matrices $\mathcal{F} = \{F_1, ..., F_{|\mathcal{C}|} \}$ where $F_i \in \mathbb{R}^{m_i \times p_i}$ such that the output feedback strategy $\vct{u}_i = F_i \vct{y}_{i} \ \ \forall  i\in \mathcal{C}$ satisfies the following: 
\begin{enumerate}[(O1)]
    \item it renders the closed-loop system stable, i.e., $\lambda(A+ \sum_{i \in \mathcal{C}} B_iF_iC_i) \in \mathbb{C}_{-}$, 
    \item it ensures dissipativity with supply function $\vct{d}^\top R\vct{d}$, where $\vct{d}=\begin{bmatrix} \vct{d}_1^\top & ... & \vct{d}_{|\mathcal{C}|}^\top\end{bmatrix}^\top$ and $R = diag(R_1, ...., R_{|\mathcal{C}|})$, and
    \item  it minimizes the upper bound difference between $f^\mathcal{F}_i$ and $f^{\mathrm{Nash}}_i \ \forall i \in \mathcal{C}$. The term $f^\mathcal{F}_i \in \mathbb{R}_{\geq0}$ is the cost function of agent $i \in \mathcal{C}$ that results from using $\vct{u}_i = F_i\vct{y}_i$ in \eqref{eq_functional_player}, and $f^{\mathrm{Nash}} _i\in \mathbb{R}_{\geq0}$ is the cost function of agent $i \in \mathcal{C}$ that results from using a Nash equilibrium strategy with complete state information and no disturbance. 
\end{enumerate}

\section{Control Synthesis}
First, we present a brief review on how to  determine a Nash equilibrium solution in NLQGs \cite{lineargames_engwerda, thulasi_constructive_nash}. In these games, each agent seeks to minimize its individual objective function in \eqref{eq_functional_player} subject to the system dynamics in \eqref{eq_mas_dyn1}-\eqref{eq_mas_dyn2} where $C_i=I_n$ and $D_i = 0_{p_i \times m_i}$ $\forall i \in \mathcal{C}$. To find a Nash equilibrium solution, it requires solving the Hamilton-Jacobi-Isaacs partial differential equations, which in NLQGs, they reduce to solving the set of CAREs for $P_i$:
\begin{align}
\Big(
{A}- \sum_{j \in \mathcal{C}_{-i} } & B_jR_j^{-1}B_j^\top P_{j} \Big)^\top P_{i} + P_{i} \Big({A} - \sum_{j \in \mathcal{C}_{-i} }   B_jR_j^{-1}B_j^\top P_{j} \Big) \notag \\ 
&- P_{i} B_iR_i^{-1}B_i^\top P_{i} + Q_{i}  = 0_{n \times n}, \forall i \in \mathcal{C}, 
\label{eq_coupled_ricky}
\end{align}
where the set $\mathcal{C}_{-i} = \{1, ..., i-1, i+1, ..., |\mathcal{C}| \}$. The set of matrices $P_i \succeq 0 \ \forall i \in \mathcal{C}$ enables the determination of the linear state-feedback Nash seeking strategies:
\begin{equation} \label{eq_linear_feed_nash}
    \vct{u}^{\mathrm{Nash}}_{i} =-R^{-1}_{i}B^\top_{i}P_{i}\vct{x}, \ \forall i \in \mathcal{C}. 
\end{equation}
Using these Nash strategies, the cost function of an agent $i \in \mathcal{C}$ is calculated as: $f^{\mathrm{Nash}}_i = \vct{x}_0^\top P_i \vct{x}_0$. Note that $\vct{u}^{\mathrm{Nash}}_{i} \ \forall i \in \mathcal{C}$ uses complete state information of the MAS. Moreover, $\vct{u}^{\mathrm{Nash}}_{i} \ \forall i \in \mathcal{C}$ has been determined without considering any disturbance in the system. In practice, in a real-world MAS, it is idealistic to implement $\vct{u}^{\mathrm{Nash}}_{i}$ and to attain the cost $f^{\mathrm{Nash}}_i \ \forall i \in \mathcal{C}$. However, the lemmas presented in this work demonstrate that the set of matrices $\{P_1, ..., P_{|\mathcal{C}|} \}$ from solving \eqref{eq_coupled_ricky} can be used as input data to synthesize output-feedback feedback controllers of the form $\vct{u}_i = F_i \vct{y}_i$, which uses partial information  $\vct{y}_i$ of MAS. First, we present Lemma \ref{lemma_stabilizing}.



\begin{lemma}
\label{lemma_stabilizing}
 Let the dynamics of a MAS be described by \eqref{eq_mas_dyn1}-\eqref{eq_mas_dyn2} with the characteristics (C1)-(C5) defined in Section \ref{sec_problem_statement}, if a set of matrices $\mathcal{F} = \{ F_1, ..., F_{|\mathcal{C}|} \}$ for the construction of $\vct{u}_i = F_i\vct{y}_i \ \forall i \in \mathcal{C}$ satisfies: 
\begin{align}
    \begin{bmatrix}
        \Gamma_i + \Upsilon_i & P_i \bar{B}_i & P_i\bar{D}_{-i}  \\
        \bar{B}^\top_i P_i & R_i & 0 \\
        \bar{D}_{-i}^\top P_i & 0 & \bar{R}_{-i}
    \end{bmatrix} \succ 0  \text{ for at least one $i \in \mathcal{C}$,} \label{eq_synthesisF}
\end{align}
then the set of matrices $\mathcal{F}$ satisfies (O1) and (O2), where for all $i \in \mathcal{C}$: 
\begin{itemize}
    \item $\Gamma_i = Q_{i} - P_{i} \left( \sum\limits_{j \in \mathcal{C}_{-i} }  B_jR_j^{-1}B_j^\top P_j + \sum\limits_{ j \in \mathcal{C}}  B_jF_jC_j \right)^\top$ \\
    $- \left( \sum\limits_{j \in \mathcal{C}_{-i} }  P_jB_jR_j^{-1} B_j^\top + \sum\limits_{ j \in \mathcal{C}}  C_j^\top F_j^\top B_j^\top \right) P_i$,
    \item $\Upsilon_i = P_i B_i \left(F_i D_i R_i^{-1} + R_i^{-1}D_i^\top F_i^\top \right) B_i^\top P_i$,
    \item $\bar{B}_i = B_i(I_{m_i} + F_i D_i)$, 
    \item $\bar{D}_{-i} = $ \\
       $\text{\small $[B_1F_1D_1 \ ... \ B_{i-1}F_{i-1}D_{i-1} \ \ B_{i+1}F_{i+1}D_{i+1} \ ... \ B_{|\mathcal{C}|}F_{|\mathcal{C}|}D_{|\mathcal{C}|}] $}$
    \item $\bar{R}_{-i} =  diag(R_1, ..., R_{i-1}, R_{i+1}, ..., R_{|\mathcal{C}|})$, and 
    \item $P_{i}$ is a solution to the CAREs \eqref{eq_coupled_ricky}.
\end{itemize}
\end{lemma}
\begin{proof}
As $R_i \succ 0$, using Schur complement formula\cite[Theorem 1.10]{Dullerud_convex_book} in \eqref{eq_synthesisF} results: 
    \begin{align} \label{eq_gamma_ups_pbr}
        & \text{$\Gamma_i + \Upsilon_i  - P_i\bar{B}_iR_i^{-1}\bar{B}_i^\top P_i$} \notag \\
        & \text{ $- \sum_{j \in \mathcal{C}_{-i} }  P_iB_jF_jD_jR_j^{-1}D_j^\top F_j^\top B_j^\top P_i \succ 0,$}         
    \end{align}
    for at least one $i \in \mathcal{C}$. Then replacing $\Gamma_i$, $\Upsilon_i$, $\bar{B}_i$ in \eqref{eq_gamma_ups_pbr} with their respective definitions, \eqref{eq_gamma_ups_pbr} becomes:
    \begin{align} \label{eq_upsi_pi}
        & \text{\small $Q_{i} - P_{i} \Bigg(\sum_{j \in \mathcal{C}_{-i} }  B_jR_j^{-1}B_j^\top P_j + \sum_{ j \in \mathcal{C}}  B_jF_jC_j \Bigg)^\top$} \notag \\
        & \text{\small $- \Bigg( \sum_{j \in \mathcal{C}_{-i} }  P_jB_jR_j^{-1} B_j^\top  + \sum_{ j \in \mathcal{C}}  C_j^\top F_j^\top B_j^\top \Bigg) P_i$}  \notag \\
        & \text{\small $- P_iB_iR_i^{-1}B_i^\top P_i  - \sum_{j \in \mathcal{C}}  P_iB_jF_jD_jR_j^{-1}D_j^\top F_j^\top B_j^\top P_i \succ 0$}, 
    \end{align}
    for at least one $i \in \mathcal{C}$. Furthermore, the set of CAREs in \eqref{eq_coupled_ricky} can be rewritten as:
\begin{align} \label{eq_cares}
    & \text{\small $ Q_i - P_iB_iR_i^{-1}B_i^\top P_i - P_i \Bigg( \sum_{j \in \mathcal{C}_{-i} }  B_j R_j^{-1} B_{j}^\top P_j \Bigg)$} \notag \\ 
    & \text{\small $- \Bigg( \sum_{j \in \mathcal{C}_{-i} }  B_j R_j^{-1} B_{j}^\top P_j \Bigg) ^\top  P_i = -(P_iA + A^\top P_i).$}
\end{align}
Using \eqref{eq_cares} in \eqref{eq_upsi_pi} results in:
    \begin{align} \label{eq_pa_closed}
        &- P_i(A+\sum\limits_{ j \in \mathcal{C}}  B_j F_j C_j) - (A + \sum\limits_{ j \in \mathcal{C}}  B_j F_j C_j)^\top P_i \notag \\
        &-    \sum_{j =1}  P_iB_jF_jD_jR_j^{-1}D_j^\top F_j^\top B_j P_i \succ 0, 
    \end{align}
    for at least one $i \in \mathcal{C}$. Next, we define the following matrices:
    \begin{align}
        A_c & = A+\sum\limits_{ j \in \mathcal{C}}  B_j F_j C_j, \label{eq_def_Ac} \\
        N  & = \begin{bmatrix}
        B_1F_1D_1 & ... & B_{|\mathcal{C}|} F_{|\mathcal{C}|} D_{|\mathcal{C}|}
    \end{bmatrix}, \label{eq_def_N} \\
        R & = diag(R_1, ..., R_{|\mathcal{C}|}). \label{eq_def_R}
    \end{align}
    Subsequently, using \eqref{eq_def_Ac}, \eqref{eq_def_N}, and \eqref{eq_def_R} in \eqref{eq_pa_closed} results in:
    \begin{align}  \label{eq_cl_p}
         & -P_i A_c - A_c^\top P_i - P_i N R^{-1} N^\top P_i \succ 0, 
    \end{align}
     for at least one $i \in \mathcal{C}$, which implies $P_i A_c + A_c^\top P_i \prec 0$, as $R \succ 0$ and $P_i \succeq 0$. By Lyapunov stability condition, $\lambda(A_c) \in \mathbb{C}_{-}$, which proves that the set of matrices $\mathcal{F}$ satisfies (O1). Note that \eqref{eq_cl_p} has to be satisfied for at least one agent $i \in \mathcal{C}$ so that the Lyapunov stability condition holds for $A_c$. 

    Next, we prove that the set of matrices $\mathcal{F}$ satisfies (O2). By Schur complement, \eqref{eq_cl_p} can be recast as
    \begin{align} \label{eq_neg_acl}
        \begin{bmatrix}
            P_i A_c + A_c^\top P_i & P_i N \\
            N^\top P_i         & -R
        \end{bmatrix} \prec 0, 
    \end{align}
    for at least one $i \in \mathcal{C}$. By the definitions in \eqref{eq_def_Ac}, \eqref{eq_def_N}, and \eqref{eq_def_R}, the closed-loop system of the MAS is represented by:
    \begin{equation} \label{eq_closed_Ac_N}
         \dot{\vct{{x}}} = A_c \vct{x} + N \vct{d}
    \end{equation}
    where $\vct{d}=\begin{bmatrix} \vct{d}_1^\top & ... & \vct{d}_{|\mathcal{C}|}^\top\end{bmatrix}^\top$. We define the storage function $V_i(\vct{x}) = \vct{x}^\top P_i\vct{x}$. Taking the Lie derivative of $V_i(\vct{x})$, using \eqref{eq_def_R}, and replacing $\dot{\vct{{x}}}$ with \eqref{eq_closed_Ac_N} results
    \begin{align}
        \dot{V}_i(\vct{x}) & = \dot{\vct{x}}^\top P_i \vct{x} + \vct{x}^\top P_i \dot{\vct{x}}  \label{eq_dev_V}\\
        & = \begin{bmatrix}
            \vct{x} \\ \vct{d}
        \end{bmatrix}^\top \begin{bmatrix}
            A_c^\top P_i+P_i A_c & P_i N \\
            N^\top P_i R & -R
        \end{bmatrix} \begin{bmatrix}
            \vct{x} \\ \vct{d}
        \end{bmatrix} + \vct{d}^\top R \vct{d}, \forall i \in \mathcal{C}. \label{eq_vdot_x_d} 
    \end{align}
Passing the term $\vct{d}^\top R \vct{d}$ to the left-hand side in \eqref{eq_vdot_x_d} and using \eqref{eq_neg_acl} results in  $\dot{V}_i(\vct{x}) < \vct{d}^\top R \vct{d}$. By the definition of dissipativity \cite{willems_dissipativity}, this inequality implies that the closed-loop system in \eqref{eq_closed_Ac_N} is dissipative for the disturbance $\vct{d}$ with the supply function $W(\vct{d}) = \vct{d}^\top R \vct{d}$. Hence, the set of matrices $\mathcal{F}$ satisfies (O2).
\end{proof}

\begin{remark}
    If there is only one agent, i.e., $\mathcal{C}=\{ 1 \}$, the NLQG reduces to a linear quadratic regulator (LQR) problem. Note that \eqref{eq_mas_dyn1}-\eqref{eq_mas_dyn2} reduce to 
     \begin{align}
            \dot{\vct{x}} & = A\vct{x} + B\vct{u}, \label{eq_dyn1_oneagent}\\
            \vct{y} & = C\vct{x} + D\vct{d}. \label{eq_dyn2_oneagent}
        \end{align}
    Moreover, the CAREs in \eqref{eq_coupled_ricky} reduce to a standard ARE. Replacing $\Gamma_i$ with its definition for $i=1$, the LMI in \eqref{eq_synthesisF} reduces to:
    \begin{align}
         \begin{bmatrix}
        Q - PBFC -C^\top F^\top B^\top P + \Upsilon & P \bar{B} \\
        \bar{B}^\top P & R
    \end{bmatrix}  \succ 0. \label{eq_lemma1_one}
    \end{align}
    where $Q$ and $R$ are the state cost and input cost matrices in \eqref{eq_functional_player}, respectively, and $P$ is the solution to the ARE. In fact, the work in \cite{Rodrigues_LQR} proposes \eqref{eq_lemma1_one} for the synthesis of output feedback control in the LQR problem. 
     It is worth mentioning that the work in \cite{Rodrigues_LQR} has been served as initial motivation for the development of Lemma \ref{lemma_stabilizing}.
\end{remark}
However, the set of matrices $\mathcal{F}$ that results from Lemma \ref{lemma_stabilizing} does not satisfy (O3). As shown in Remark \ref{rk_impor_cons} at the end of this part, the solely use of the LMI in \eqref{eq_lemma1_one} can result in a set of matrices $\mathcal{F}$ that yield unbounded values for the agents' cost functions.

In Lemma \ref{lemma_f_fnash} we provide additional constraints that prevent this issue and enables the set of matrices $\mathcal{F}$ to ultimately satisfy (O3). First, the following corollary is reviewed.

\begin{corollary}[Exercise 3.13 in \cite{Dullerud_convex_book}]
\label{cor_boundedx_dullerud}
    Consider the system $\dot{\vct{x}}= A\vct{x}+B\vct{u}$ where $\vct{x} \in \mathbb{R}^n$, $\vct{u} \in \mathbb{R}_m$, and $\lambda(A) \in \mathbb{C}_{-}$. If $\vct{u} \in L^m_2[0,\infty)$ then $\lim\limits_{t \to\infty} \vct{x}(t) = 0$.
\end{corollary}

\begin{lemma}
\label{lemma_f_fnash}
     Let $\vct{u}_i = F_i \vct{y}_i$ $\forall i \in \mathcal{C}$ in \eqref{eq_mas_dyn1}-\eqref{eq_mas_dyn2} where the set of matrices $\mathcal{F}= \{ F_1, ..., F_{|\mathcal{C}|} \}$ satisfies Lemma \ref{lemma_stabilizing}, then 
     \begin{align} 
        f^{\mathcal{F}}_i \leq f^{\mathrm{Nash}}_i + k_1 g^2_i + k_2  h^2_i + k_3 g_i h_i \notag\\
        + \sum_{ j \in \mathcal{C}_{-i}} a_j g_j +   b_j h_j, \forall i \in \mathcal{C},
        \label{eq_lemma_f_fnash}
     \end{align}
     where $f^{\mathcal{F}}_{i}$, $f^{\mathrm{Nash}}_{i}$ are defined in (O3) in Section \ref{sec_problem_statement}, $k_1$, $k_2$, $k_3$, $l_j$, $h_j$ $\forall j \in \mathcal{C}_{-i}$ are constants independent of $F_i$ and time, 
     $g_i = \norm{ R_i^{-\frac{1}{2}}B_i^\top P_i + R_i^{\frac{1}{2}} F_i C_i }_2$, and $h_i = \norm{R_i^{\frac{1}{2}} F_i D_i}_2$. 
\end{lemma}
\begin{proof}
By Lemma \ref{lemma_stabilizing}, the set of matrices $\mathcal{F}$ renders $\lambda \left(A_c \right) \in \mathbb{C}_-$. By definition in (C5), $\vct{d}_j \in L^m_2[0, \infty)$. Hence, using the fundamental theorem of calculus for the storage function $V_i(\vct{x}) = \vct{x}^\top P_i \vct{x}$ and applying Corollary \ref{cor_boundedx_dullerud} to \eqref{eq_closed_Ac_N} results in:
     \begin{equation} \label{eq_limxp}
        \int_0^\infty \dot{V_i}(\vct{x})dt =  \lim_{t \to \infty} \vct{x}_{(t)}^\top P_i \vct{x}_{(t)} -\vct{x}_0^\top P_i \vct{x}_0 = -\vct{x}_0^\top P_i \vct{x}_0, 
        \forall i \in \mathcal{C}.
     \end{equation}
Using $\vct{u}_i = F_i \vct{y}_i$ $\forall i \in \mathcal{C}$, \eqref{eq_mas_dyn2}, and \eqref{eq_limxp} in \eqref{eq_functional_player} results in:
\begin{align} \label{eq_cost_dt}
    f^{\mathcal{F}}_i & = \vct{x}_0P_i\vct{x}_0 + \int_0^{\infty} \{ \dot{V}_i(\vct{x}) + \vct{x}^\top Q_i \vct{x} \notag \\
    & + (F_iC_i\vct{x} + F_iD_i\vct{d}_i)^\top R_i (F_iC_i\vct{x} + F_iD_i\vct{d}_i) \} dt, \forall i \in \mathcal{C}.
\end{align}
Replacing $\dot{V}_i(\vct{x})$ with \eqref{eq_dev_V}, and using \eqref{eq_coupled_ricky} and \eqref{eq_closed_Ac_N} in \eqref{eq_cost_dt} results that $\forall i \in \mathcal{C}$:
    \begin{align} \label{eq_cost_f_unfolded}
        & \text{\small $f^\mathcal{F}_i  = \vct{x}_0P_i\vct{x}_0 + \int_0^\infty \vct{x}^\top P_i B_i R_i^{-1}B_i^\top P_i \vct{x}$} \notag \\
        & \text{\small $+ 2 \vct{x}^\top P_i \Bigg( \sum_{j \in \mathcal{C}_{-i} }  B_jR_j^{-1}B_j^\top P_j \Bigg) \vct{x} + \vct{x}^\top C_i^\top F_i^\top R_i F_i C_i \vct{x}$} \notag \\
        &\text{\small $+ 2 \vct{x}^\top C_i^\top F_i^\top R_i F_i D_i \vct{d}_i + \vct{d}_i^\top D_i^\top F_i^\top R_i F_i D_i \vct{d}_i$} \notag \\
        &\text{\small $+2\vct{x}^\top \Bigg( \sum_{ j \in \mathcal{C}} B_j F_j C_j \Bigg)^\top P_i \vct{x} + 2 \Bigg(\sum_{ j \in \mathcal{C}} B_j F_j D_j \vct{d}_j   \Bigg)^\top P_i \vct{x} \ dt.$}
    \end{align}
    We define $\vct{p}_i = R^{-\frac{1}{2}}_{i}B_i^\top P_i \vct{x}$ and  $\vct{p}_{ij} = R_j^{-\frac{1}{2}}B_j^\top P_i \vct{x}$, note that $\vct{p}_i = \vct{p}_{ii}$. Moreover,  
    $\vct{q}_i = R^{\frac{1}{2}}_i F_i C_i \vct{x}$ and $\vct{r}_i = R_i^{\frac{1}{2}} F_i D_i \vct{d}_i$. Using these definitions in \eqref{eq_cost_f_unfolded} results in:
    \begin{align} 
        f^\mathcal{F}_i & = \vct{x}_0P_i\vct{x}_0 + \int_{0}^\infty \{ \vct{p}_i^\top\vct{p}_i+2\sum_{j \in \mathcal{C}_{-i} } \vct{p}_{ij}^\top \vct{p}_j + \vct{q}_i^\top\vct{q}_i + 2\vct{q}_i^\top\vct{r}_i  \notag \\
        & + \vct{r}_i^\top \vct{r}_i + 2 \sum_{ j \in \mathcal{C}}  \vct{q}_i^
        \top \vct{p}_{ij} + 2 \sum_{ j \in \mathcal{C}}  \vct{r}_i^
        \top \vct{p}_{ij} \} dt, \forall i \in \mathcal{C}.\label{eq_f_pqr}
    \end{align}
     Using the property that $f^{\mathrm{Nash}}_i = \vct{x}_0^\top P_i \vct{x}_0 \ \forall i \in \mathcal{C}$, grouping terms and using the triangle inequality and the Cauchy–Schwarz inequality in \eqref{eq_f_pqr} results in: 
    \begin{align}
        \text{\small $f^{\mathcal{F}}_i$} & = \text{\small $f^{\mathrm{Nash}}_i + \int_0^\infty  \norm{\vct{p}_i + \vct{q}_i + \vct{r}_i}^2_2 \ dt$} \notag \\
        & \text{\small $+ 2\int_0^\infty \sum_{j \in \mathcal{C}_{-i}  }  | \vct{p}_{ij}^\top(\vct{p}_j + \vct{q}_j + \vct{r}_j) | \ dt$}, \label{eq_fnas_f_still} \\
        & \text{\small $\leq f^{\mathrm{Nash}}_i + \int_0^\infty \norm{\vct{p}_i+\vct{q}_i}^2_2 + \norm{\vct{r}_i}^2_2 + 2\norm{\vct{p}_i + \vct{q}_i}_2 \ \norm{\vct{r}_i}_2$} \notag \\
        & \text{\small $+ 2 \sum_{j \in \mathcal{C}_{-i}  }  \int_0^\infty \norm{\vct{p}_{ij}}_2 \norm{\vct{p}_j + \vct{q}_j + \vct{r}_j}_2 \ dt, \forall i \in \mathcal{C}$}. \label{eq_perf_bound}
    \end{align}
As $\lambda ( A_c ) \in \mathbb{C}_- $ and $\vct{d}_i \in L^{m_i}_{2(0, \infty)} \ \forall i \in \mathcal{C}$, $ \int_0^\infty \norm{\vct{x}}_2 \ dt$ and $ \int_0^\infty \norm{\vct{x}}^2_2 \ dt$ are bounded above by a constant independent of $F_i$ and time.  Moreover, the induced matrix norm property states that $\norm{A\vct{x}} \leq \norm{A} \norm{x}$. Using these properties and the definitions of $\vct{p}_i$ and $\vct{q}_i$ results in:
\begin{align}
    \int_0^\infty \norm{\vct{p}_i+ \vct{q}_i}^2_2 dt & =  \int_0^\infty \norm{\left( R_i^{-\frac{1}{2}}B_i^\top P_i + R_i^{\frac{1}{2}} F_i C_i \right) \vct{x}}^2_2 dt, \\ 
    & \leq k_1\norm{ R_i^{-\frac{1}{2}}B_i^\top P_i + R_i^{\frac{1}{2}} F_i C_i }^2_2, \forall i \in \mathcal{C}. \label{eq_bo_pq}
\end{align}
Similarly, using the definition of $\vct{r}_i$ results in:
\begin{align} 
    & \text{\small $\int_0^\infty \norm{\vct{r}_i}_2^2 dt =  \int_0^\infty \norm{R_i^{\frac{1}{2}} F_i D_i \vct{d}_i}_2^2 dt \leq k_2 \norm{R_i^{\frac{1}{2}} F_i D_i}_2^2, \forall i \in \mathcal{C}$} \label{eq_bo_r} \\
    & \text{\small $\int_0^\infty \norm{\vct{p}_i + \vct{q}_{i}}_2 \norm{\vct{r}_i} dt  \leq k_3 \norm{ R_i^{-\frac{1}{2}}B_i^\top P_i + R_i^{\frac{1}{2}} F_i C_i }_2$} \notag \\
    & \ \ \ \ \ \ \ \ \ \ \ \  \ \ \ \ \ \ \ \ \ \ \ \ \ \ \ \ \ \ \ \ \ \ \ \text{\small $\norm{R_i^{\frac{1}{2}} F_i D_i}_2, \forall i \in \mathcal{C}$} \label{eq_bo_pq_pro}.
\end{align}
Similarly, by the triangle inequality, 
\begin{align} \label{eq_bo_inner}
    & \text{\small $ \sum_{j \in \mathcal{C}_{-i}  } \int_0^\infty \norm{\vct{p}_{ij}}_2  \norm{\vct{p}_j + \vct{q}_j  + \vct{r}_j}_2 \ dt$} \\
    & \text{\small $\leq \sum_{j \in \mathcal{C}_{-i}  } \int_0^\infty \norm{R_j^{-\frac{1}{2}}B_j^\top P_i \vct{x}}_2 \left( \norm{\vct{p}_j + \vct{q}_j }_2 + \norm{\vct{r}_j}_2\right) dt,$} \\
    & \text{\small $ \leq \sum_{j \in \mathcal{C}_{-i}  } a_j\norm{ R_j^{-\frac{1}{2}}B_j^\top P_j + R_j^{\frac{1}{2}} F_j C_j }_2 + b_j \norm{R_j^{\frac{1}{2}} F_j D_j}_2,$} \notag \\ 
    & \ \ \ \ \ \text{$\forall i \in \mathcal{C},$}
    \end{align}
    where $k_1$, $k_2$, $a_j$, $b_j$ are constants independent of $F_i$ and time.
Finally, using \eqref{eq_bo_pq}, \eqref{eq_bo_r}, \eqref{eq_bo_pq_pro} and \eqref{eq_bo_inner} in \eqref{eq_perf_bound} results \eqref{eq_lemma_f_fnash}.
\end{proof}

\begin{remark}
    In the idealistic case, where there is perfect state information ($C_i=I_n$ $\forall i \in \mathcal{C}$) and no disturbance ($D_i = 0_{p_i \times m_i}$ $\forall i \in \mathcal{C}$), \eqref{eq_fnas_f_still} reduces to:
    \begin{align}
        & \text{\small $f^{\mathcal{F}}_i = f^{\mathrm{Nash}}_i + \int_0^\infty   \Bigl\{ \norm{R^{-\frac{1}{2}}_{i}B_i^\top P_i \vct{x} +  R^{\frac{1}{2}}_i F_i C_i \vct{x}}^2_2 $} \notag \\
        & \text{\small $+ 2 \sum_{j \in \mathcal{C}_{-i}  }   ( R_j^{-\frac{1}{2}}B_j^\top P_i \vct{x})^\top(R^{-\frac{1}{2}}_{j}B_j^\top P_j + R^{\frac{1}{2}}_j F_j C_j \vct{x}) \Bigl\} \ dt, \forall i \in \mathcal{C}$},\label{eq_remark1_fnash}
    \end{align}
    where the definitions of $\vct{p}_i$, $\vct{q}_i$, $\vct{r}_i$, and $\vct{p}_{ij}$ have been used. In \eqref{eq_remark1_fnash}, note that $f^{\mathcal{F}}_i = f^{\mathrm{Nash}}_i$ when $F_i = R^{-1}_{i}B_i^\top P_i$ $\forall i \in \mathcal{C}$, which is indeed the standard Nash state-feedback strategy presented in \cite{Engwerda_LQbook, thulasi_constructive_nash, capello_nash_approx}. Hence, Lemma \ref{lemma_f_fnash} indicates that the quadratic cost function $f_i^\mathcal{F}$ is equal to the quadratic cost function $f_i^\mathcal{F}$ for the idealistic setting (state-feedback with no disturbance), i.e., $C_i=I_n$ and $D_i = 0_{p_i \times m_i}$ $\forall i \in \mathcal{C}$.
\end{remark}
\begin{corollary}[Section 4.6.3 in \cite{boyd_optimization}]
\label{cor_matrixmin_boyd}
    Let $A(\vct{x}) = A_0 + x_1 A_1 + ... + {x}_n A_n$, where $A_i \in \mathbb{R}^{p \times q}$ and $\vct{x} = \begin{bmatrix}
        x_1 & ... & x_n
    \end{bmatrix}^\top \in \mathbb{R}^n$. The unconstrained problem $\min\limits_{\vct{x}} \norm{A(\vct{x})}_2$ is equivalent to the SDP model: $\min\limits_{\vct{x}, \ t} t$ subject to $\begin{bmatrix}
        t I_p & A(\vct{x}) \\
        A(\vct{x})^\top  & t I_q
    \end{bmatrix} \succeq 0$, where $t \in \mathbb{R}$.
\end{corollary}

\begin{lemma}
\label{lemma_sdp_problem}
We are given the parameters $a$, $b$, and $c \in \mathbb{R}_{>0}$, and the matrices $\Gamma_i$, $\Upsilon_i$, $P_i$, $\bar{B}_{i}$, $\bar{D}_{-i}$ and $\bar{R}_{-i}$ are as defined in Lemma \ref{lemma_stabilizing}. If the optimal solution of the following SDP problem:
 \begin{align}
       & \ \ \ \ \ \ \ \ \ \min_{\substack{\alpha_i, \beta_i, \gamma_i, F_i \\ \forall  \ i \in \mathcal{C}}}  \sum _{i \in \mathcal{C}} -a \alpha_i + b \beta_i + c \gamma_i, \label{eq_sdp_obj} \\
      & \text{subject to: }\notag \\ 
     & \begin{bmatrix}
        \Gamma_i + \Upsilon_i & P_i \bar{B}_i & P_i\bar{D}_{-i}  \\
        \bar{B}^\top_i P_i & R_i & 0 \\
        \bar{D}_{-i}^\top P_i & 0 & \bar{R}_{-i}
    \end{bmatrix}  \succeq \alpha_i I_{n+m} \  \ \forall  \ i \in \mathcal{C}, \label{eq_lemma1_c} \\
    & \begin{bmatrix}
        \beta_i I_{m_i} & R_i^{-\frac{1}{2}}B_i^\top P_i + R_i^{\frac{1}{2}} F_{i} C_{i} \\
        ( R_i^{-\frac{1}{2}}B_i^\top P_i + R_i^{\frac{1}{2}} F_{i} C_{i})^\top & \beta_i I_{n}
    \end{bmatrix} \succeq 0 \notag \\
    &  \  \ \ \ \ \ \ \ \ \ \ \ \ \ \ \ \ \ \ \ \ \ \ \ \ \ \ \ \ \ \ \ \ \ \ \ \ \ \ \ \ \ \ \ \ \ \ \ \ \ \ \ \forall  \ i \in \mathcal{C}, \label{eq_nashapprox}\\
     & \begin{bmatrix}
        \gamma_i I_{m_i} & R_i^{\frac{1}{2}} F_i D_i \\
        (R_i^{\frac{1}{2}} F_i D_i)^\top & \gamma_i I_{m_i}
    \end{bmatrix} \succeq 0 \  \ \forall  \ i \in \mathcal{C}, \label{eq_atten}
 \end{align}
 returns $\alpha_i>0$ for at least one $i \in \mathcal{C}$, then the resulting set of matrices $\mathcal{F}$ for the output feedback control $\vct{u}_i = F_i \vct{y}_i \  \forall i \in \mathcal{C}$ in \eqref{eq_mas_dyn1}-\eqref{eq_mas_dyn2} satisfies (O1), (O2), and (O3).
\end{lemma}
\begin{proof}
    If $\alpha_i > 0$ for at least one $i \in \mathcal{C}$, then the set of matrices $\mathcal{F}$ satisfies the LMI in \eqref{eq_synthesisF} of Lemma \ref{lemma_stabilizing}.
    Using Corollary \ref{cor_matrixmin_boyd}, note that $\min\limits_{\substack{\beta_i, F_i \text{ for } i \in \mathcal{C}}} \sum_{i \in \mathcal{C}}  b\beta_i$ subject to \eqref{eq_nashapprox} is equivalent to $\min\limits_{\substack{\substack{\beta_i, F_i  \ \forall i \in \mathcal{C}}}} \sum_{i \in \mathcal{C}}  bg_i$, where $g_i$ is defined in Lemma \ref{lemma_f_fnash}. Likewise, $\min\limits_{\substack{\gamma_i, F_i \ \forall i \in \mathcal{C}}} \sum_{i \in \mathcal{C}}  c\gamma_i$ subject to \eqref{eq_atten} is equivalent to $\min\limits_{\substack{\substack{\gamma_i, F_i  \ \forall i \in \mathcal{C}}}} \sum_{i \in \mathcal{C}} ch_i$, where $h_i$ is defined in Lemma \ref{lemma_f_fnash}. Therefore, the SDP model in \eqref{eq_sdp_obj}-\eqref{eq_atten} is equivalent to
    \begin{align}
               & \min_{\substack{\alpha_i, \beta_i, \gamma_i, F_i \\ \forall  \ i \in \mathcal{C}}}  \sum _{i \in \mathcal{C}} -a \alpha_i +  bg_i + c h_i, \label{eq_sdp_obj2}
    \end{align}
    subject to \eqref{eq_lemma1_c}. By Lemma 2, note that the previous optimization problem minimizes the upper-bound of $f^\mathcal{F}_i - f^{\mathrm{Nash}}_i$ $\forall$ $i \in \mathcal{C}$. Therefore, the set of matrices $\mathcal{F}$ that solves the SDP model in \eqref{eq_sdp_obj}-\eqref{eq_atten} satisfies O1, O2, O3.
\end{proof}
\begin{remark}  
If the solution to the SDP model in \eqref{eq_sdp_obj}-\eqref{eq_atten} returns  $\alpha_i < 0$ $\forall \ i \in \mathcal{C}$, it does not necessarily imply that the resulting 
set of matrices $\mathcal{F}$ should be discarded. The reason is that the left-hand (LHS) side of \eqref{eq_lemma1_c} can be still positive definitive. Thus, we suggest checking post-solution if this is case. If the LHS of \eqref{eq_lemma1_c} is still positive definite, then the set of matrices $\mathcal{F}$ that results from solving the SDP model satisfies (O1), (O2), and (O3).
\end{remark}
\begin{remark}  
The constraint \eqref{eq_lemma1_c} implies that $\bar{R} - \alpha_i I_{m} \succeq 0$ $\forall \ i \in \mathcal{C}$, where $\bar{R} = diag(R_1, ..., R_{|\mathcal{C}|})$. Therefore, $ \min\limits_{\alpha_i \ \forall i \in \mathcal{C}}\sum_{i \in \mathcal{C}}  -a\alpha_i$, where $a \in \mathbb{R}_{>0}$, is bounded  below. By Corollary \ref{cor_matrixmin_boyd}, $\beta_i>0$ and $\gamma_i>0$ $\forall \ i \in \mathcal{C}$, then $\min\limits_{\beta_i, \gamma_i \ \forall i \in \mathcal{C}}\sum_{i \in \mathcal{C}}  b\beta_i + c\gamma_i$ are bounded below. Therefore the objective function  \eqref{eq_sdp_obj} is bounded below.
\end{remark}

\begin{remark}
\label{rk_impor_cons}
The absence of constraints \eqref{eq_nashapprox} and \eqref{eq_atten} in the SDP model \eqref{eq_sdp_obj}-\eqref{eq_atten} can result in a set of matrices $\mathcal{F}$ that yields an unbounded cost function $f^\mathcal{F}_i \ \forall i \in \mathcal{C}$. To demonstrate this, consider the SDP model omitting constraints \eqref{eq_nashapprox} and \eqref{eq_atten}:
\begin{align}
    & \ \ \ \ \ \ \ \ \ \ \ \ \ \ \ \ \ \ \ \ \ \ \ \  \min_{\substack{\alpha_i, F_i \\ \forall  \ i \in \mathcal{C}}}  \sum _{i \in \mathcal{C}} -a \alpha_i, \label{eq_sdp_obj_unbounded} \\
      & \text{subject to: }\notag \\ 
     & \begin{bmatrix}
        \Gamma_i + \Upsilon_i & P_i \bar{B}_i & P_i\bar{D}_{-i}  \\
        \bar{B}^\top_i P_i & R_i & 0 \\
        \bar{D}_{-i}^\top P_i & 0 & \bar{R}_{-i}
    \end{bmatrix}  \succeq \alpha_i I_{n+m} \  \ \forall  \ i \in \mathcal{C}. \label{eq_lemma1_c_unbounded} 
\end{align}
     Let the dynamics of a single agent system be given by \eqref{eq_dyn1_oneagent}-\eqref{eq_dyn2_oneagent}. By Schur complement and Woodbury matrix identity $(R_1-\alpha I_{m})^{-1} = R^{-1} + R^{-1}(R-\alpha I_m)^{-1}$ and considering $D=0$, \eqref{eq_lemma1_c} becomes:
        \begin{align} \label{eq_one_case_exam}
            -(A+BFC)^\top P  & - P(A+BFC) \succeq \notag \\
            & \alpha I_n + PBR^{-1}(R-\alpha I)^{-1}B^\top P.
        \end{align}
    For the scalar case when $A=B=C=Q=R=1$, as $R - \alpha \geq 0$ by \eqref{eq_lemma1_c_unbounded}, \eqref{eq_sdp_obj_unbounded} is minimized when $\alpha = R = 1$. Moreover, the solution to the ARE is $P = 2.41$. Replacing $A, B, C, R, P, \alpha$ in \eqref{eq_one_case_exam} with their respective values, \eqref{eq_one_case_exam} results in:
     \begin{align}
            -2(2.41)(1 + F) \geq 1. \label{eq_f_scalar_case}
        \end{align}
     Note that the value of $F$ can be infinitely negative to satisfy \eqref{eq_f_scalar_case}. Replacing $F$ with an infinitely negative value in \eqref{eq_functional_player} yields an unbounded cost function value. Therefore, constraints \eqref{eq_nashapprox} and \eqref{eq_atten} that result from Lemma \ref{lemma_f_fnash} enable to produce a bounded cost function  $f^\mathcal{F}_i \ \forall i \in \mathcal{C}$. 
     \end{remark}
     \begin{remark}
                Constraints
              \eqref{eq_nashapprox} and \eqref{eq_atten} do not only enable to produce bounded cost function $f^\mathcal{F}_i \ \forall i \in \mathcal{C}$ as we demonstrated in Remark \ref{rk_impor_cons}. Indeed, they serve to approximate $f^\mathcal{F}_i$ to $f^{\mathrm{Nash}}_i \ \forall i \in \mathcal{C}$. Moreover, note that the role of  constraint \eqref{eq_atten} is to attenuate the impact of the disturbance $\vct{d}_i \ \forall i \in \mathcal{C}$ on the system.
     \end{remark}

\section{Implementation and Case Studies}

\subsection{Introduction to an open-source toolbox}
\label{sec_toolbox} 
    We have created CMAS, a toolbox for the Control of Multi-agent Systems. CMAS is available in Python and in MATLAB. Both packages   
    implement the framework of this article and other novel frameworks we are continuously developing for the control of MAS \cite{model_free_paper}. 
    
    Despite the wide range of control packages offered in Python and MATLAB, to the best of our knowledge, there is no package for the control of MAS. CMAS aims to fill this gap by offering a comprehensive suite of functions and examples for the control of MAS integrating a game-theoretic perspective. CMAS is publicly available in the GitHub repository \cite{cmas_repo}. As follows, we provide a succinct overview of the essential functions in CMAS that are associated to this work: 
\begin{enumerate}
    \item \lstinline|mas|\_\lstinline|output|\_\lstinline|feedback|:
It solves the SDP model of Lemma \ref{lemma_sdp_problem}. In Python, we use CVXPY \cite{diamond2016cvxpy, agrawal2018rewriting}, whereas in MATLAB, we use YALMIP \cite{yalmip}. As input data, the function reads the matrices $A$, $B_i$, $C_i$, $D_i \ \forall i \in \mathcal{C}$ from \eqref{eq_mas_dyn1}-\eqref{eq_mas_dyn2}, $Q_i$, $R_i \ \forall i \in \mathcal{C}$ from \eqref{eq_functional_player}, the set of matrices $P_i \ \forall i \in  \mathcal{C}$ that solves the CAREs in \eqref{eq_coupled_ricky}, and the parameters $a, b$ and $c$ from \eqref{eq_sdp_obj}. The function returns the set of matrices $F_i \ \forall i \in  \mathcal{C}$.
    \item \lstinline|newton|\_\lstinline|raphson|\_\lstinline|cares|: It returns the set of matrices $P_i \ \forall i \in \mathcal{C}$ that solves the CAREs in \eqref{eq_coupled_ricky}. To this end, we propose solving a SDP model iteratively that incorporates linear matrix equations. These matrix equations are derived from applying the Newton-Raphson method to the CAREs in \eqref{eq_coupled_ricky}. These matrix equations are presented in algorithm $7$ in \cite{Engwerda2007}. As input data, the function reads the matrices $A$, $B_i$, $Q_i$, $R_i$, $\ \forall i \in \mathcal{C}$. When these matrices are highly ill-conditioned or badly scaled, the solver for the SDP model can experience numerical issues. As an alternative to this function, we present \lstinline|lyapunov|\_\lstinline|iteration|\_\lstinline|cares|.
    
    \item \lstinline|lyapunov|\_\lstinline|iteration|\_\lstinline|cares|: It returns the set of matrices $P_i \ \forall i \in \mathcal{C}$ that solves the CAREs \eqref{eq_coupled_ricky}. To this end, it implements the algorithm $5$ presented in \cite{Engwerda2007}. This algorithm solves standard AREs iteratively. We use functions in Python and MATLAB that balance and re-scale the input matrices to prevent numerical issues. However, a limitation of this algorithm is that it requires that each pair $(A, B_i) \ \forall i \in \mathcal{C}$ to be stabilizable \cite{li_lyapunoviter}. This limitation is not present in \lstinline|newton|\_\lstinline|raphson|\_\lstinline|cares|.
\end{enumerate}
It should be noticed that neither \lstinline|newton|\_\lstinline|raphson|\_\lstinline|cares| outperforms the \lstinline|lyapunov|\_\lstinline|iteration|\_\lstinline|cares| nor vice versa as each function has its own benefits and limitations. Moreover, note that 
first \lstinline|newton|\_\lstinline|raphson|\_\lstinline|cares| or \lstinline|lyapunov|\_\lstinline|iteration|\_\lstinline|cares| should be executed, then the resulting set of matrices $P_i \ \forall i \in \mathcal{C}$ should be inputted to \lstinline|mas|\_\lstinline|output|\_\lstinline|feedback|. In the repository, each function incorporates a detail description of its inputs and outputs. CMAS offers more functions, besides the three detailed earlier.

In Section \ref{sec_validation}, the proposed control framework is compared with other frameworks for MAS control. Through the comparison in numerical cases, we can validate and showcase the virtues of the proposed control framework. In Section \ref{sec_powergrid}, using the proposed framework we design a control scheme to regulate the electrical frequency in a power grid. Through this case study, we showcase the scalability of the control framework in a high-order and realistic networked system. All case studies are also available in the repository \cite{cmas_repo}. All numerical results shown in this article are generated using CMAS in MATLAB. These results are virtually the same as using CMAS in Python.

\subsection{Validation and comparison with other frameworks}
\label{sec_validation}
The framework in \cite{Rodrigues_LQR} proposes the synthesis of an output feedback control scheme for the LQR problem. This corresponds to a MAS with one agent. 
We compare our proposed framework against the framework in \cite{Rodrigues_LQR} for the case in \cite[Example 4.1]{Rodrigues_LQR}. Table \ref{table_comparison_works}  illustrates that the values of $F_1$ and $f^{\mathcal{F}}_1$ that result from using \lstinline|mas|\_\lstinline|output|\_\lstinline|feedback| are very close to those that result from using \cite{Rodrigues_LQR}, differing only in up to $-0.0014\%$. However, note that, in contrast to \lstinline|mas|\_\lstinline|output|\_\lstinline|feedback|, the framework in \cite{Rodrigues_LQR} does not incorporate the constrains \eqref{eq_nashapprox} and \eqref{eq_atten}. The absence of these constraints can yield unbounded cost values in some cases as it is illustrated in Remark \ref{rk_impor_cons}.

The work in \cite{capello_distributednash} proposes the synthesis of distributed controllers by solving sequential AREs for a MAS with heterogeneous agents. The mentioned work considers the following case study. Suppose a set of agents $\mathcal{C} = \{1,2,3,4\}$ that moves in the plane $\mathbb{R}^2$ with restrictive communication, where each agent has single-integrator dynamics as follows:
\begin{align}
    \dot{\vct{x}}_i & = B_{ii}\vct{u}_i + \sum_{j \in \mathcal{N}_i} B_{ij} \vct{u}_{j}, \label{eq_cap_case_1}\\
    \vct{y}_i & = \begin{bmatrix}
        \vct{x}^\top_{j_1} & ... & \vct{x}^\top_{j_{|\mathcal{N}_i|}}
    \end{bmatrix}^\top \label{eq_cap_case_2}, 
\end{align}
where $x_i \in \mathbb{R}^2$ represents the position of agent $i \in \mathcal{C}$ on the plane, $B_{ii} = I_{2}$, $B_{ij} = -I_2$, and $\mathcal{N}_1 =\{ 1 \}$, $\mathcal{N}_2 =\{1, 2\}$, $\mathcal{N}_3 =\{2, 3\}$, and $\mathcal{N}_3 =\{2, 4\}$. The objective is to design a control scheme that steers the position of the agents to satisfy a desired formation on the plane. To this end, first, we synthesize the set of matrices $F_i \ \forall i \in \mathcal{C}$ using the control framework proposed in \cite{capello_distributednash}. Second, we re-formulate \eqref{eq_cap_case_1}-\eqref{eq_cap_case_2} in the form \eqref{eq_dyn1_oneagent}-\eqref{eq_dyn2_oneagent} and synthesize $F_i \ \forall i \in \mathcal{C}$ using \lstinline|mas|\_\lstinline|output|\_\lstinline|feedback|. Table \ref{table_comparison_works} shows that \lstinline|mas|\_\lstinline|output|\_\lstinline|feedback| returns exactly the same set of matrices $F_i \ \forall i \in \mathcal{C}$, and consequently, the same cost function values $f_i^{\mathcal{F}} \ \forall i \in \mathcal{C}$ that the framework in \cite{capello_distributednash} returns. Hence, the control framework that we propose has been validated with the work in \cite{capello_distributednash}. However, in contrast to our proposed control framework, the framework in \cite{capello_distributednash} is restricted to MAS with dynamically uncoupled agents. The previous case is indeed an example of this type of MAS because the differential equation \eqref{eq_cap_case_1} for each agent $i$ does not have the state $\vct{x}_j$ of other agents $j \neq i $ on its right-hand side.

\begin{table}[]
\setlength{\tabcolsep}{0.1em}
\caption{Comparison of resulting set of matrices $\mathcal{F}$ and cost function values $f^{\mathcal{F}}_i$ for different case studies using the proposed control framework and other frameworks.}
\begin{tabular}{ccccccc}
\toprule
                       & \multicolumn{2}{c}{ \cite[Example 4.1]{Rodrigues_LQR}} & \multicolumn{2}{c}{\cite[Section VI]{capello_distributednash}} & \multicolumn{2}{c}{ \cite[Section V-A]{distributednash_roy}} \\ \midrule
                       & {[}26{]}                                & Proposed                                & {[}19{]}                                & Proposed                               & \cite{distributednash_roy}          & Proposed          \\ \midrule \midrule
 $F_1$     & $\begin{bmatrix}
    6.90 \\
    84.92
\end{bmatrix}^\top$                                         &   $\begin{bmatrix}
    6.89 \\
    84.95
\end{bmatrix}^\top$                                        &          $-0.1 I_2$                             &    $-0.1 I_2$                                      &    $\begin{bmatrix}
    -0.89 \\
    -1.32 \\
    -0.50 \\
    0.63 \\
    0.18
\end{bmatrix}^\top$           &       $\begin{bmatrix}
    -0.99 \\
    -1.64 \\
    0.63 \\
    0.60 \\
    0.13
\end{bmatrix}^\top$            \\ 
   $f^{\mathcal{F}}_1$ &     21.12                                    &   21.09                                      &        0.14                                 &  0.14                                      &      1.65         & 1.83                   \\ \midrule
 $F_2$               &      -                                   &      -                                   &        $\begin{bmatrix}
     -0.45 I_2 \\
     -0.08 I_2
 \end{bmatrix}^\top $                                 &  $\begin{bmatrix}
     -0.45 I_2 \\
     -0.08 I_2
 \end{bmatrix}^\top $                                   &  $\begin{bmatrix}
     -0.89 \\
     -1.32 \\
     0.50 \\
     0.63 \\
     0.18
 \end{bmatrix}^\top$                 &     $\begin{bmatrix}
     -0.98\\
     -1.69\\
     0.66\\
     0.62\\
     0.14
 \end{bmatrix}^\top$              \\
   $f^{\mathcal{F}}_2$ &     -                                    &       -                                  &     2.66                                    &    2.66                                    &   1.22            &        1.34           \\ \midrule
 $F_3$               &      -                                   &          -                               &    $\begin{bmatrix}
     -0.45 I_2 \\
     -0.22 I_2
 \end{bmatrix}^\top$                                     &   $\begin{bmatrix}
     -0.45 I_2 \\
     -0.22 I_2
 \end{bmatrix}^\top$                                 &   $\begin{bmatrix}
     0.50 \\
     0.61 \\
     0.50 \\
     0.61 \\
     -1.32 \\
     -1.70 \\
     0.16\\
     0.16
 \end{bmatrix}^\top$                &      $\begin{bmatrix}
     0.33 \\
     0.37 \\
     0.28 \\
     0.17 \\
     -0.96 \\
     -1.84 \\
     0.10\\
     0.10
 \end{bmatrix}^\top$              \\
   $f^{\mathcal{F}}_3$ &   -                                      &          -                               &             3.11                            &    3.11                                    &  1.81             &  1.74                 \\ \midrule
 $F_4$               &         -                                &      -                                   &      $\begin{bmatrix}
     -0.45 I_2 \\
     -0.22 I_2
 \end{bmatrix}^\top$                                     &     $\begin{bmatrix}
     -0.45 I_2 \\
     -0.22 I_2
 \end{bmatrix}^\top$                                     &    $\begin{bmatrix}
     0.63 \\
     0.30 \\
     0.61 \\
     0.28\\
     -1.34
 \end{bmatrix}^\top$           &    $\begin{bmatrix}
     0.61 \\
     0.25 \\
     0.64 \\
     0.26 \\
     -1.36
 \end{bmatrix}^\top$               \\
   $f^{\mathcal{F}}_4$ &    -                                     &    -                                     &          3.11                             &    3.11                                    &    0.64           &   0.61                \\ \midrule
 $F_5$               &         -                                & -                                        &         -                                &         -                               &         $\begin{bmatrix}
       0.63\\
       0.30\\
       0.61\\
       0.28\\
       -1.34
   \end{bmatrix}^\top$          &  $\begin{bmatrix}
       0.62 \\
       0.22 \\
       0.64 \\
       0.26 \\
       -1.36
   \end{bmatrix}^\top$                 \\
   $f^{\mathcal{F}}_5$ &              -                           &           -                              &             -                            &     -                                   &     0.47      &   0.46            \\ \bottomrule
\end{tabular}
\label{table_comparison_works}
\end{table}

Finally, the proposed control framework is compared with the framework in \cite{distributednash_roy}. As it is described in Section \ref{sec_introduction}, the mentioned work develops an iterative algorithm with LMIs to synthesize distributed control for MAS. The case study in\cite[Section V-A]{distributednash_roy} considers a similar formation control problem as \cite{capello_distributednash} does, but the MAS has five agents and double-integrator dynamics. The system states are the position $p_i \in \mathbb{R}$ for agents for $i=\{1, ..., 5\}$ and the velocity $v_i \in \mathbb{R}$ for agents $i=\{1, 2, 3\}$. Table \ref{table_comparison_works} shows that \lstinline|mas|\_\lstinline|output|\_\lstinline|feedback| returns cost function values that differ in up to $11\%$ with respect to the framework in  \cite{distributednash_roy} for this case study. This is probably because \cite{distributednash_roy} considers an alternative definition to Nash equilibrium. Moreover, the four top panels in Figure \ref{fig_our_vs_roy} illustrate that the 
settling time across all the states is approximately $5$ seconds using the proposed control, while it is $10$ seconds using the control in \cite{distributednash_roy}. The four top panels in Figure \ref{fig_our_vs_roy} also describe that the overshoot is on average up to $10\%$ greater using the control in \cite{distributednash_roy} than using the proposed control. These panels also display that the proposed control attains less oscillating response across all states. This is because the proposed control places the dominant eigenvalue on the real axis as the bottom panel of Figure \ref{fig_our_vs_roy} shows, thus providing more damping to the system than the control in \cite{distributednash_roy}.


\begin{figure}[h!]
      \centering
      \includegraphics[width=\linewidth]{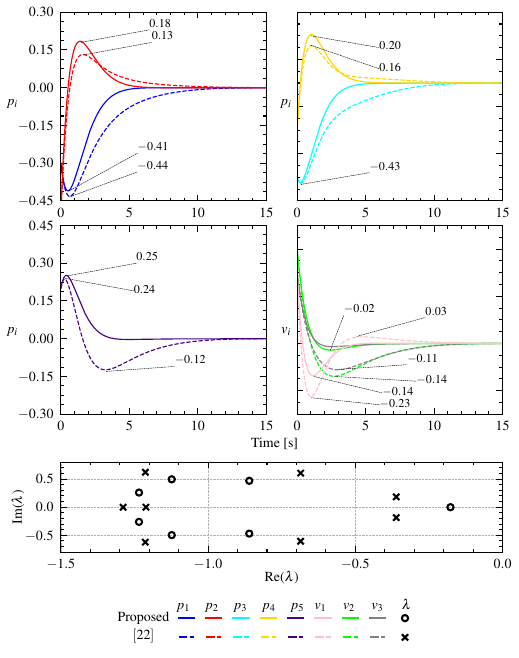} 
      \caption{Comparison of time-domain performance of the states $p_i \in \mathbb{R}$ (position) and $v_i \in \mathbb{R}$ (velocity) and eigenvalues $\lambda$ that result from using the proposed framework and the framework in \cite{distributednash_roy} for the case study in \cite[Section V-A]{distributednash_roy}. } 
      \label{fig_our_vs_roy}
\end{figure}

\subsection{Distributed frequency regulation in a power grid}
\label{sec_powergrid}
Using the proposed control framework, this case designs a distributed frequency regulation strategy for a modified version of the IEEE 24-bus power grid. As Figure \ref{fig_24_ieee_system} shows, this 230-kV system incorporates  11 grid-forming inverters (GFMIs) that have virtual inertia \cite{sohail_paul_gfmi}, 5 grid-following inverters (GFLIs) that have phase locked loop \cite{behrooz_gflis}, 34 transmission lines and 24 buses. First, we formulate a NLQG for the frequency regulation of the grid. To this end, GFMIs and GFLIs are considered selfish agents that seek to minimize their cost functions subject to the system dynamics. These system dynamics are represented by a multi-nodal frequency dynamics model, which includes the swing equation to represent GFMIs, power flow equations for the transmission lines, and equations for power dynamics of GFLIs. Refer to \cite{paulpatricia_nash_freq} for a detailed explanation on the construction of this NLQG. Next, Kron reduction is applied to eliminate the algebraic equations in the multi-nodal frequency dynamics model. Finally, the resulting state-space model has 40 states, including electrical frequency of GFMIs, voltage angles of buses connected to a GFMI, and states of the GFLIs, and 29 inputs, including loads and power references of GFMIs and GFLIs. Aside from the higher complexity of this model, this case study substantially differs from the case studies of Section \ref{sec_validation} in the presence of heterogeneous agents (GFMIs and GFLIs) and dynamical couplings (transmission lines) between the agents.  

\begin{figure}[t]
      \centering
      \includegraphics[width=\linewidth]{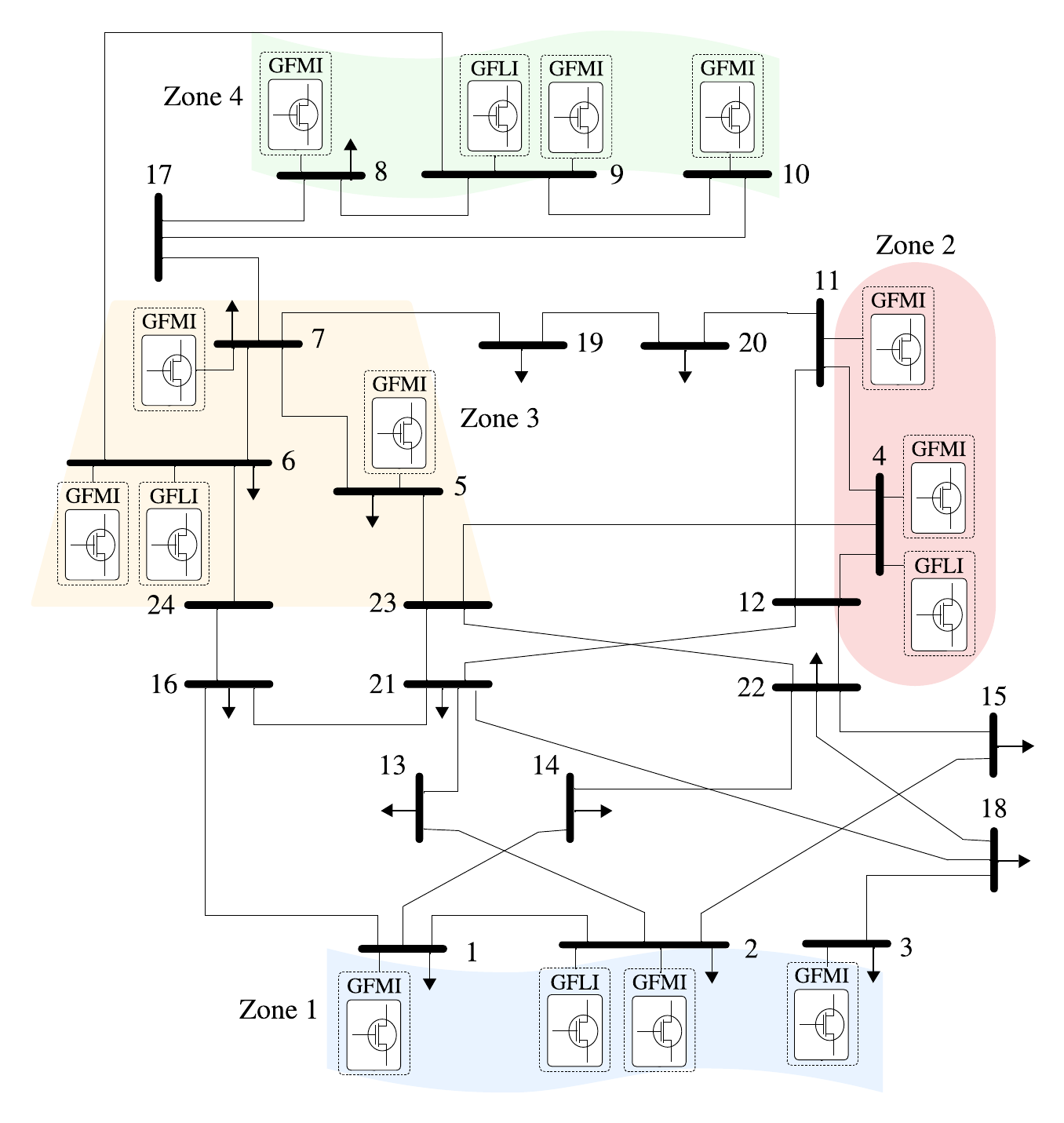} 
      \caption{Modified IEEE 24-bus power system. GFMIs and GFLIs can share state information only within their zone.}
      \label{fig_24_ieee_system}
\end{figure}

Next, matrices $A$ and $B_i$, $Q_i$ and $R_i \ \forall i \in \mathcal{C}$ are inputted into \lstinline|lyapunov|\_\lstinline|iteration|\_\lstinline|cares| to compute the Nash equilibrium solution. In a Mac Studio, Apple M2 Ultra, $128$ GB, \lstinline|lyapunov|\_\lstinline|iteration|\_\lstinline|cares| is set to $60$ iterations and takes approximately $10$ seconds to finish with an error of $10^{-4}$. The function \lstinline|newton|\_\lstinline|raphson|\_\lstinline|cares| can also be used instead of \lstinline|lyapunov|\_\lstinline|iteration|\_\lstinline|cares|. This function is set to $3$ iterations and takes $211$ seconds to finish with an error of $10^{-4}$.

The resulting matrices $P_i  \ \forall i \in \mathcal{C}$ are then inputted to \lstinline|mas|\_\lstinline|output|\_\lstinline|feedback|. As input data to this function, the matrix $C_i$ for agent $i \in \mathcal{C}$ is chosen such that agent $i$, GFMI or GFLI, has access to the state information, including electrical frequency and state of the power dynamics of GFLIs, only within its corresponding zone. Moreover, agents cannot access voltage angles of the buses in any circumstance. The execution of the function \lstinline|mas|\_\lstinline|output|\_\lstinline|feedback| takes $0.6$ seconds to complete.

The resulting set of static matrices $\mathcal{F}$ constitutes the distributed frequency regulation control scheme for the grid. As explained before, agents exclusively use some states of their neighbors, without needing the full system's state. We implement and run this control scheme in a high-fidelity model of the  grid that we developed in Simulink. This high-fidelity model is nonlinear and has 261 states that describe the EMT dynamics of the GFMIs, GFLIs and the network. Hence, this model intends to provide a realistic representation of the dynamics of voltage, frequency, and control schemes that integrate the components of the grid. Refer to STING \cite{sting_repo} for a more detail explanation of the control schemes of the GFMIs and GFLIs in this high-fidelity model.

\begin{figure}[h!]
      \centering
      \includegraphics[width=\linewidth]{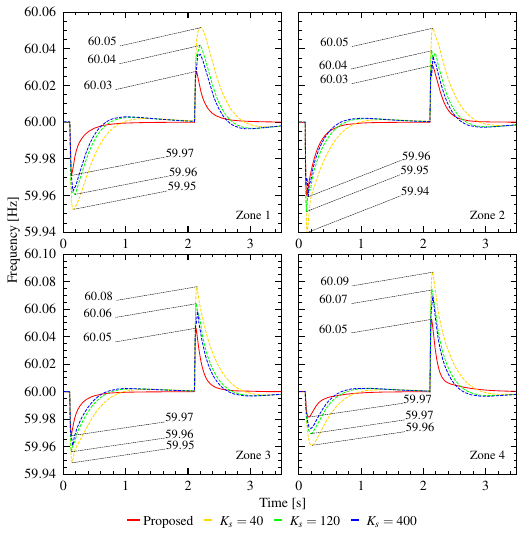} 
      \caption{Comparison of time-domain performance of  frequency (Hz) regulation that results from using the proposed control framework and three decentralized integral action strategies for the grid in Figure \ref{fig_24_ieee_system}. In each panel, the frequency for each control strategy is the frequency center of inertia (COI) of the zone. } 
      \label{fig_freq_grid}
\end{figure}

To contrast the proposed control framework, we implement three different frequency regulation strategies in the grid. These strategies consist of an integral action with gain $K_s \in \mathbb{R}$ that is incorporated in the active power control loop of a GFMI in each zone. In this manner, these strategies intend to model a decentralized secondary frequency regulation in the grid. Figure \ref{fig_freq_grid} shows that the frequency deviation is on average up to $41\%$ less when using the proposed control as opposed to using the decentralized strategies. Figure \ref{fig_freq_grid} also displays that the proposed control is able to set the frequency to its nominal value in approximately $1$ second, which is two times faster than the other strategies. Aside from this performance comparison, note that the proposed control enables the participation of  GFLIs and GFMIs from a game-theoretic perspective in a frequency regulation setting. 


\section{Conclusion}
We introduce a novel SDP model that synthesizes a set of static matrices, which are used as output feedback gains for agents in MAS. The LMI constraints in the SDP model have been derived such that the resulting control enables the stabilization and dissipativity of the MAS. We theoretically demonstrate that we can use the Nash equilibrium solution associated to state-feedback strategies to construct output feedback strategies. The synthesized control ensures the stability and dissipativity of the closed-loop system. Aside from theoretical contributions, we also showcase practical contributions. Compare to prior work which usually validates in small-scale numerical linear models, we experimentally validate our framework in a high-fidelity nonlinear power system model of 261 states. This showcase that our proposed SDP problem can be solved for large-scale problems using off-the-shelf solvers. This is a virtue compared to prior frameworks which their scalability remains unexplored. Moreover, we implement our theoretical framework  as an open-source toolbox in Python and MATLAB.

Although the computation of the solution for the SDP problem is centralized, the implementation in realistic MAS illustrates that the control scheme is distributed, as each agent's control strategy requires partial and local state information. As a future work, it may be appealing to investigate the theoretical connection between the existence of CAREs solution and the SDP feasibility in
Lemma~\ref{lemma_sdp_problem}. Additionally, future work may investigate the theoretical sufficient and necessary conditions for the existence of set of matrices $\mathcal{F}$ before the proposed SDP problem is solved. However, this issue is indeed posing a more general question: whether the static output feedback control for MAS exist. Moreover, the SDP is model-based as it requires to know the state-space model of the MAS. Currently, we are working to re-formulate the SDP model in a data-driven fashion.



\bibliography{bibliography.bib}

\begin{IEEEbiography}[{\includegraphics[width=1in,height=1.25in,clip,keepaspectratio]{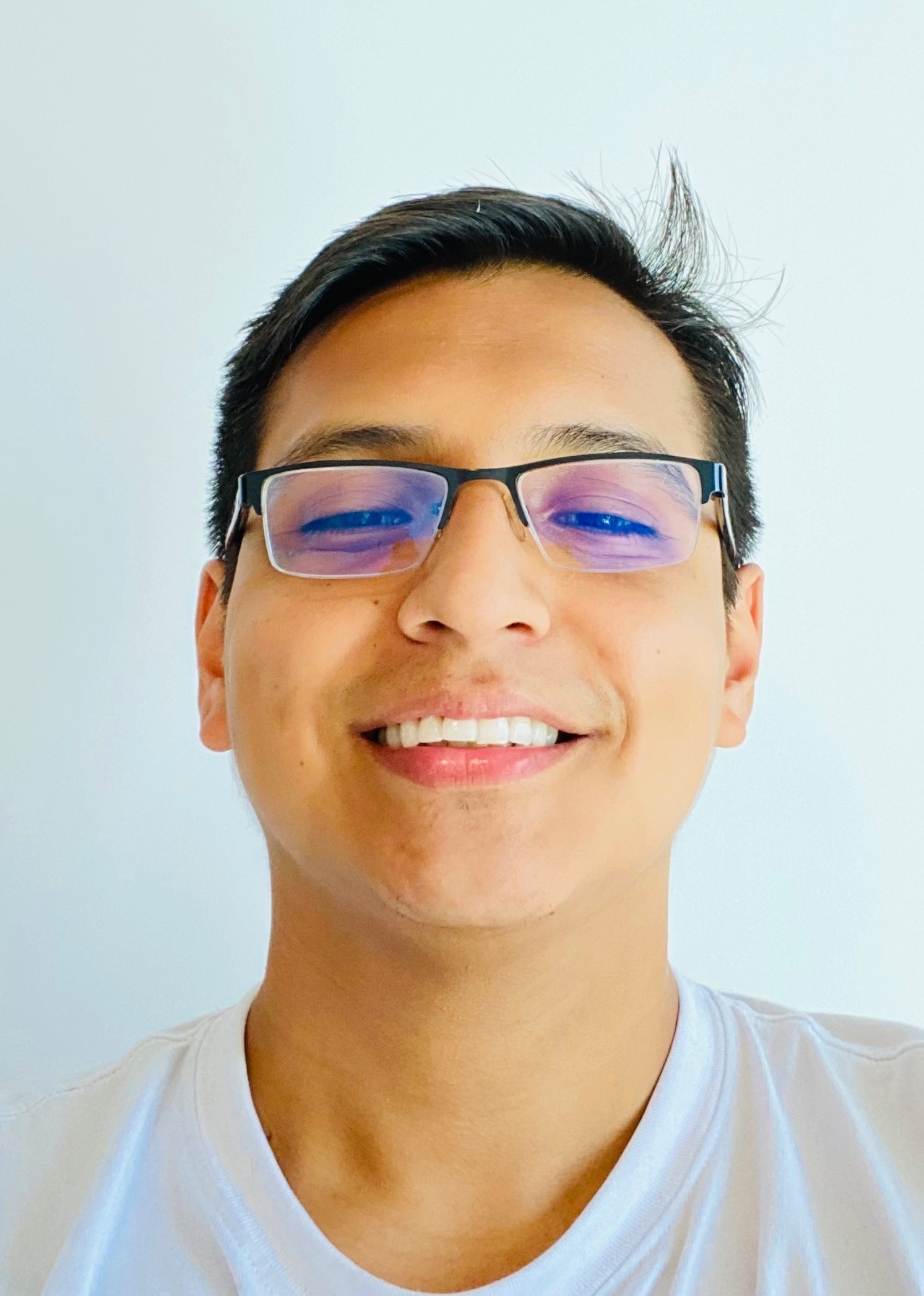}}]
{Paul Serna-Torre}
(Member, IEEE) received the B.Sc. degree in electrical engineering from Universidad Nacional de Ingenieria, Peru, where he received the Eduardo de Habich award for achieving the highest academic performance in the Electrical Engineering program. He received the M.A. degree in Applied Mathematics from University of California San Diego, USA. Paul has five years of experience in industry, working on contingency analysis, electricity markets and the coordination for the maintenance of substations and generators in the Independent System Operator (COES) and in the Energy Commission (Osinergmin) in Peru. Moreover, he worked on the design of power transformers in Delcrosa (electromechanics company) in Peru. He is currently working towards the Ph.D. degree at University of California San Diego, USA, with the supervision of professor Patricia Hidalgo-Gonzalez. He has been an intern at Control Group at Pacific Northwest National Laboratory (PNNL). Paul has been awarded   the Bicentenary Generation Scholarship by the Ministry of Education of Peru, and the Opportunity Funds scholarship by EducationUSA. Paul's research work focuses on the intersection of dynamical systems, game theory and learning techniques for the modeling of electromagnetic transient dynamics and the capacity expansion of power systems.
\end{IEEEbiography}
\begin{IEEEbiography}[{\includegraphics[width=1in,height=1.25in,clip,keepaspectratio]{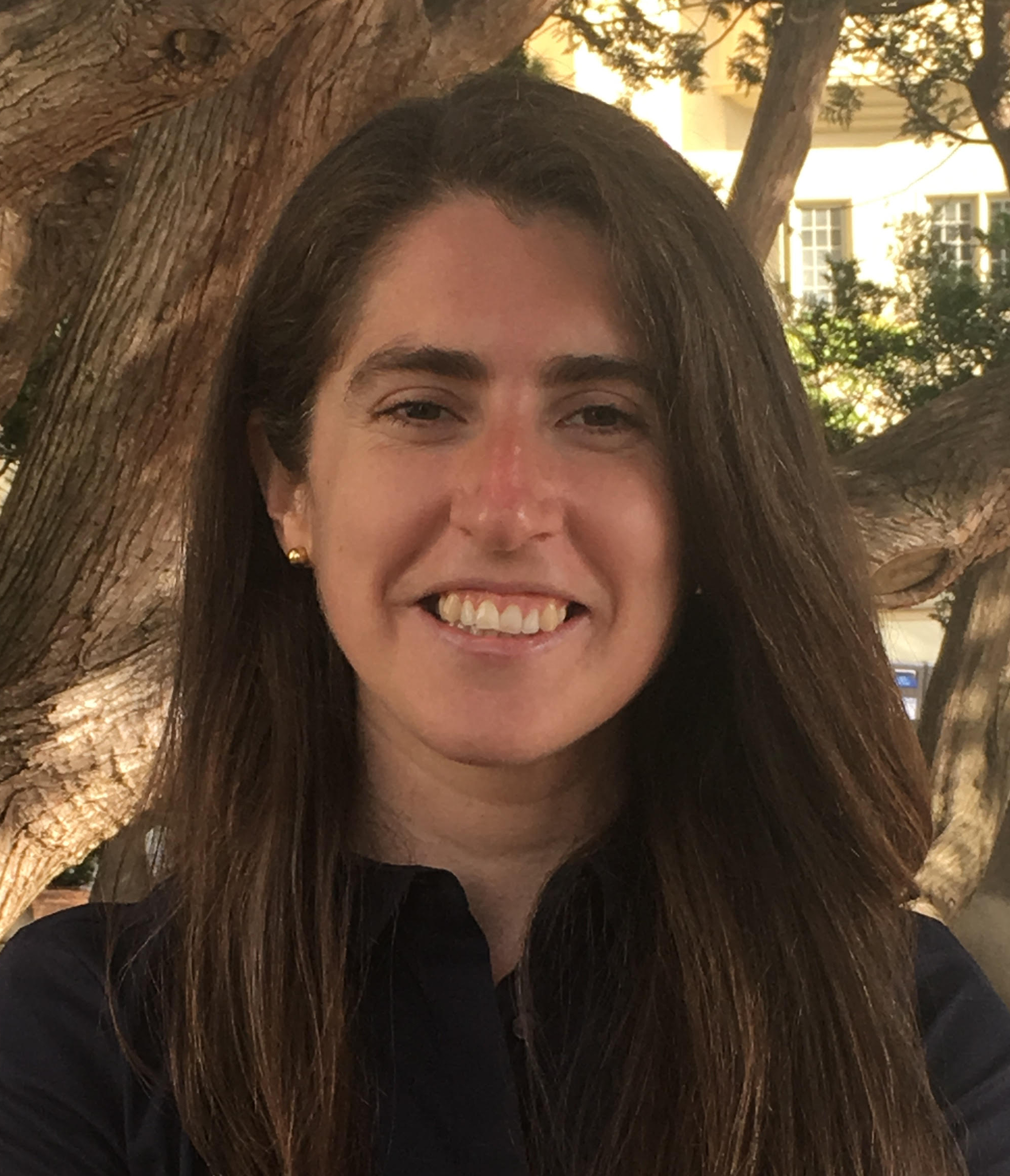}}]{Patricia Hidalgo-Gonzalez}
(Member, IEEE) received the B.Sc. degree in industrial and electrical engineering from Pontificia Universidad Catolica of Chile, and two M.Sc. degrees from the University of California, Berkeley in Energy and Resources and Electrical Engineering and Computer Sciences. She is currently an assistant professor at University of California San Diego. Her research work focuses high penetration of renewable energy using optimization, control theory and machine learning. Prof. Hidalgo-Gonzalez is the 2025 Institute of Industrial and Systems Engineers Energy Systems Division Outstanding Young Investigator Awardee, an NSF GRFP fellow, Siebel Scholar in Energy, Rising Star in Electrical Engineering and Computer Science, received the Outstanding Graduate Student Instructor Award at UC Berkeley (for teaching Convex Optimization), has been awarded Best paper at the Power Systems Computation Conference 2020, and is part of the Editors’ Highlights in Nature Communications (best 50 papers recently published in “Engineering and Infrastructure”), her work was cited in The White House 2024 Economic Report of the President, among other recognitions. She is part of the IEEE Task forces ''Data-driven controls for distributed systems'' and ''Wildfire Aware-Risk Management (WARM) for the Electric Power Grid'' and has served as best paper judge in the IEEE Control Systems Society Technical Committee on Energy Systems. She has also served as the opening talk for the energy storage session in EPRI's 43rd Seminar on Resource Planning in Washington D.C., and as a speaker at a symposium from the U.S. National Academy of Sciences and the German National Academy of Sciences Leopoldina in Halle, Germany, among others.
\end{IEEEbiography}

\end{document}